\documentclass[12pt]{amsart}
\usepackage{txfonts}
\usepackage{soul}
\usepackage{amscd,amssymb,mathabx,mathtools}
\usepackage[arrow,matrix,graph,frame,poly,arc,tips]{xy}
\usepackage[dvipsnames]{xcolor}
\usepackage{graphicx}
\usepackage{calrsfs}
\usepackage[labelfont=rm]{subcaption}
\usepackage{tikz}
\usetikzlibrary{decorations.markings}
\usetikzlibrary{shapes}
\usetikzlibrary{backgrounds}
\usetikzlibrary{calc}
\usetikzlibrary{shapes.misc, positioning}
\usepackage{mdwlist}
\usepackage{xfrac}
\DeclareCaptionSubType*{figure}

\usepackage{multirow}
\usepackage{enumerate}
\usepackage{verbatim}
\usepackage{comment}
\usepackage{todonotes}

\DeclarePairedDelimiter{\form}{\langle}{\rangle}

    \newcommand\ba{\begin{align*}}
    \newcommand\ea{\end{align*}}
    \newcommand\be{\begin{enumerate}}
    \newcommand\ee{\end{enumerate}}
    \newcommand\bpp{\begin{prop}}
    \newcommand\epp{\end{prop}}
    \newcommand\bpb{\begin{prob}}
    \newcommand\epb{\end{prob}}
    \newcommand\bd{\begin{defn}}
    \newcommand\ed{\end{defn}}
    \newcommand\bh{\begin{hint}}
    \newcommand\eh{\end{hint}}

    \newcommand\bN{\mathbb{N}}
    
    \newcommand\bR{\mathbb{R}}
    
    \newcommand\bZ{\mathbb{Z}}

    \newcommand\diam{\operatorname{diam}}

    \DeclareMathOperator\Homeo{Homeo}

    \newcommand\sse{\subseteq}
    \newcommand\co{\colon}

    \DeclareMathOperator\Diff{Diff}
    \DeclareMathOperator\Lip{Lip}

    \def\thetitle{Smoothing countable group actions on metrizable spaces}
    \def\theauthors{Inhyeok Choi, Sang-hyun Kim}
    \usepackage{hyperref}
    \hypersetup{
    	colorlinks=false,
    	plainpages,
    	urlcolor=black,
    	linkcolor=black
    	pdftitle=   \thetitle,
    	pdfauthor=  {\theauthors}
    }

    \theoremstyle{plain}
    \newtheorem{thm}{Theorem}[section]

    \newtheorem{cor}[thm]{Corollary}
    \newtheorem{prop}[thm]{Proposition}
    \newtheorem{con}[thm]{Conjecture}
    \newtheorem{que}[thm]{Question}
    
    \newtheorem*{claim*}{Claim}

    \theoremstyle{remark}

    \theoremstyle{definition}
    \newtheorem{defn}[thm]{Definition}
    \newtheorem{prob}{Problem}[section]

\begin{document}
    \title\thetitle
    \date{\today}
    
    \author[I. Choi]{Inhyeok Choi}
    \email{inhyeokchoi48@gmail.com}
\urladdr{https://inhyeokchoi48.github.io/}    

    \author[S. Kim]{Sang-hyun Kim}
    \email{kimsh@kias.re.kr}
    \urladdr{https://www.kimsh.kr}
    
\address{School of Mathematics, Korea Institute for Advanced Study (KIAS), Seoul, 02455, Korea}

\maketitle

\begin{abstract}
We prove that every topological action of a countable group on a metrizable space can be realized as a bi-Lipschitz action with respect to some compatible metric. 
This extends a result due to U. Hamenst\"{a}dt regarding finitely generated groups, and our proof is based upon her idea. This also gives a simple proof of a theorem due to Deroin, Kleptsyn and Navas regarding one-manifolds.
We also establish an analogous result for closed subgroups of locally compact groups. 
\end{abstract}

\section{Introduction}\label{sec:intro}

The homeomorphism group $\Homeo{M}$ and the diffeomorphism group $\Diff^k{M}$ of a mainfold $M$ are groups with rich structures, which can depend heavily on the regularity $k\in\bN$ of the group. For instance, there exists a faithful $C^k$ action of the lamplighter group $\bZ\wr \bZ$ on $[0,1]$ that cannot be topologically conjugated into a $C^{k+1}$ action; see~\cite{Tsuboi1987, CC1988} and also~\cite[Section 8.1.2]{KK2021book}. 
In fact, for every real $r\ge1$, there exists a finitely generated $C^r$--diffeomorphism group acting on $[0,1]$ that admits no group theoretic embedding into $\bigcup_{s>r}\Diff^s[0,1]$; see~\cite{KK2020crit}.

Meanwhile, such a rigidity does not generalize to low regularity. 
Deroin, Kleptsyn and Navas showed in~\cite[Th{\'e}or{\`e}me D]{DKN2007} that every topological action of a countable group on a compact connected one-manifold can be topologically conjugated into a bi-Lipschitz action.
For instance, the group 
\[
\form{a,b,c\mid a^2=b^3=c^7=abc}\]
does embed into the bi-Lipschitz homeomorphism group of the interval, while it does not embed into $\Diff^1[0,1]$; this latter fact was observed by Thurston~\cite{Thurston1974Top}.
Hamenst{\"a}dt recently gave a simpler proof of ~\cite[Th{\'e}or{\`e}me D]{DKN2007} for a finitely generated group using an averaging technique~\cite{Hamenstadt2024PAMS}.
We denote by $\Lip(X,\delta)$ the group of bi-Lipschitz homeomorphisms of a metric space $(X,\delta)$.
In this note, we employ her strategy and establish the following.

\begin{thm}\label{thm:top}
If $X$ is a metrizable space and if $\Gamma$ is a countable subgroup of $\Homeo(X)$, then $X$ admits a compatible metric $\delta$ such that $\Gamma\le\Lip(X,\delta)$.
\end{thm}
We emphasize that the finitely generated case of the above is proved in~\cite{Hamenstadt2024PAMS}. The idea of embedding a countable group into $F_{2}$ originates from~\cite{higman1949embedding}.

\begin{proof}[Proof of Theorem~\ref{thm:top}]
We let $F_r$ denote the free group of rank $1\le r\le\infty$.
We have an embedding 
\[F_\infty=\form{x_0,x_1,x_2,\ldots}\hookrightarrow F_2=\form{x,t}\] defined by
$x_i\mapsto t^i x t^{-i}$.
Let $\|\cdot\|$ denote the word length in $F_2$ with the basis $\{x,t\}$, which naturally restricts to $F_\infty$ by the above embedding.

Let us pick an arbitrary compatible metric $\hat\delta$ on $X$; we may further assume this metric is bounded, by considering a compatible metric $\hat\delta/(1+\hat\delta)$ if necessary.
Since $\Gamma$ is countable, we have a representation
\[
\rho\co F_\infty\to \Homeo(X)\]
such that $\Gamma=\rho(F_\infty)$.
It suffices to prove that $\rho$ factors as
\[
F_\infty\stackrel{\rho}{\to}\Lip (X,\delta)\hookrightarrow \Homeo(X)\]
for some compatible metric $\delta$.

We construct $\delta$ as a weighted average of the $\rho$--translates of $\hat\delta$ as follows:
\[
\delta
:=
\sum_{g\in F_\infty}
\exp(-s \|g\|) \hat\delta\circ \rho(g)^{-1}\]
Here, we fix $s>\log 3$.
Let us establish the following.

\begin{claim*}
\begin{enumerate}[(i)]
    \item\label{p:conv} The infinite sum defining $\delta$ is convergent.
    \item\label{p:met} The map $(x,y)\mapsto \delta(x,y)$ determines a bounded metric on $X$.
    \item\label{p:compat} The topology of $X$ is compatible with the metric $\delta$.
    \item\label{p:lip} We have $\rho(F_\infty)\le \Lip (X,\delta)$.
\end{enumerate}
\end{claim*}
Parts~(\ref{p:conv}) and~(\ref{p:met}) follow from that $\hat\delta$ is bounded
and that 
\[\sum_{g\in F_\infty} \exp(-s\|g\|)\le\sum_{g\in F_2}\exp(-s\|g\|)\le\sum_{r=0}^\infty \#\{g\in F_2\co \|g\|\le r\} \exp(-sr)<\infty.\]

We let $B(x,r)$ and $\hat B(x,r)$ denote the radius $r$ open balls centered at $x\in X$
in the metrics $\delta$
and $\hat \delta$, respectively.
For part (\ref{p:compat}), it suffices to show that for all $x\in X$ and $r>0$ there exists $r'$ such that
\begin{equation}\label{eq:incl}\tag{A}
\hat B(x,r')\sse B(x,r).
\end{equation}
Denoting by $\diam(X,\hat\delta)$ the diameter of the metric space $(X,\hat\delta)$, 
we can pick a finite set $A\sse F_\infty$ such that 
\[\sum_{g\not\in A} \exp(-s\|g\|) \diam(X,\hat\delta)<r/2.\]
We then pick $r'>0$ such that for all $a\in A$ and for all $y\in \hat B(x,r')$ we have
\[ \hat\delta(\rho(a^{-1})(x),\rho(a^{-1})(y))<\frac{r/2}{\#A}.\]
The inclusion relation in \eqref{eq:incl} is now straightforward.

For part (\ref{p:lip}), it suffices to note from the triangle inequality for $\|\cdot\|$ that
\[
\exp(-s\|g\|)\delta(x,y)\le\delta(\rho(g)(x),\rho(g)(y))\le \exp(s\|g\|)\delta(x,y).\]
The claim is proved, and the conclusion is now immediate.
\end{proof}

Since every pair of compatible metrics on connected one--manifolds are topologically conjugate to each other, we immediately recover the aformentioned theorem of Deroin, Kleptsyn and Navas.
\begin{cor}
If $X$ is a connected one--manifold (with or without boundary),
then every countable subgroup of $\Homeo(X)$ is topologically conjugate into $\Lip (X)$.
\end{cor}

An \emph{Oxtoby--Ulam measure} on a manifold is a fully supported, atomless, null-on-boundary, Borel probability measure on the manifold. Recall two measures $\mu$ and $\nu$ on a space $X$ are \emph{bi-Lipschitz equivalent} if the ratio of two measures are bounded by $C$ and by $1/C$ for some $C\ge1$. 
For a measure $\mu$, we let $[\mu]$ denote its bi-Lipschitz equivalence class. We denote by $\Homeo_{[\mu]}(M)$ the group of homeomorphisms on $X$ that preserves the equivalence class $[\mu]$.
The following is a generalization of the above corollary to higher dimensions, based on the proof of~\cite[Corollary 2.5]{Hamenstadt2024PAMS}.

\begin{cor}
If $M$ is a compact connected $n$--manifold (possibly with boundary) equipped with an Oxtoby--Ulam measure $\mu$, then every countable subgroup of $\Homeo(M)$ is topologically conjugate into $\Homeo_{[\mu]}(M)$.
\end{cor}

\begin{proof}
Let us fix a surjective representation 
\[
\rho\co F_\infty\to \Gamma\le\Homeo(M).\]
We define a new measure
\[
\nu:=\sum_{g\in F_\infty} \exp(-s\|g\|)\mu\circ g^{-1}.\]
Then $\nu$ is also an Oxtoby--Ulam measure, satisfying for each $g\in \Gamma$ that
\[
\exp(-s\|g\|)\nu\le g^*\nu\le \exp(s\|g\|)\nu.\]
A classical theorem due to von Neumann and to Oxtoby--Ulam states that an Oxtoby--Ulam measure on a compact connected manifold is unique up to topological conjugacy.
This implies that $\mu$ and $\nu$ are topologically conjugate, 
and hence that each element in some fixed conjugate of $\Gamma$  preserves the bi-Lipschitz equivalence class of $\mu$. The conclusion is immediate.
\end{proof}

A similar argument generalizes to locally compact groups; note that the corollary below is a direct generalization of Theorem~\ref{thm:top} under the additional hypothesis of local compactness on the space.

\begin{cor}\label{cor:topVar}
If $X$ is a locally compact metrizable space, if $H$ is a closed subgroup of a compactly generated locally compact group, and if $\rho : H \rightarrow \Homeo(X)$ is a continuous homomorphism, then $X$ admits a compatible metric $\delta$ such that $\rho$ factors through a continuous homomorphism $H\to\Lip (X,\delta)$.
\end{cor}

We note that, contrary to the discrete setting where every countable group embeds in a finitely generated group~\cite{higman1949embedding}, not every $\sigma$-compact locally compact group sits in a compactly generated locally compact group as a closed subgroup~\cite{caprace2014on-embeddings}.

\begin{proof}
Let $\iota : H \rightarrow G$ be the embedding of $H$ into a compactly generated locally compact group $G$. Let $K$ be a compact symmetric neighborhood of $1\in G$ that generates $G$. Let $\mu$ be the Haar measure on $G$, normalized by $\mu(K) = 1$. As in the proof of Theorem \ref{thm:top}, let us fix a bounded compatible metric $\hat{\delta}$ on $X$.

We define a simplicial graph $\Gamma = (V, E)$ as follows. The set $V$ is a maximal subset of $G$ such that $1\in V$ and such that every pair of distinct elements $x, y \in V$ satisfy $x^{-1} y \notin K$. In particular, we have
\[
G = \bigcup_{v\in V}vK.\]
Let us draw an edge between two vertices $x, y \in V$ if and only if $x^{-1} y \in K^{3}$.

To observe the connectivity of $\Gamma$, pick an arbitrary $v \in V$. Since $K$ generates $G$, 
there exists a sequence $(g_{i})_{i=0}^{n}$ in $G$ such that $g_0=1, g_n=v$ and $g_{i}^{-1} g_{i+1} \in K$ for each $i$. For each $g_{i}$, there exists $v_{i} \in V$ such that $g_{i}^{-1} v_{i} \in K$ thanks to the maximality of $V$. It follows that $(1, v_{1}, \ldots, v_{n-1}, v)$ is a path in $\Gamma$ connecting $1$ to $v$. 

Let us take a symmetric open neighborhood $U$ of $1$ such that $U^{2} \subseteq K$. We claim that $\{vU : v \in V\}$ are disjoint. Indeed, if $vu = v'u'$ for some elements $u, u' \in U$ and a pair of distinct elements $v, v' \in V$, then $(v')^{-1} v = u'u^{-1} \in U^{2} \subseteq K$. This contradicts the definition of $V$.

Note that $K^{3}$ is a compact set with an open cover $\{x U : x \in K^{3}\}$, which hence admits a finite subcover $\{x_{1} U, \ldots, x_{N} U\}$.
If we set \[Lk(v_{0}) := \{w \in V : (v_{0}, w) \in E\},\] then $v_{0}^{-1} Lk(v_{0})$ is contained in $K^3 \subseteq \cup_{i=1}^{N} x_{i} U$. Moreover, each $x_{i} U$ can take at most one element of $v_{0}^{-1} Lk(v_{0}) \subseteq v_{0}^{-1} V$, due to the disjointness of $\{vU :v \in V\}$. Consequently, $v_{0}^{-1} Lk(v_{0})$ consists of at most $N$ elements.
Hence, the graph $\Gamma$ has a bounded valence. 
There exists some $s_{0} > 0$ such that \[
\# \{ g\in V \co  d_\Gamma(1,g)\le r\} = o(e^{s_{0}r}).
\]
Using this graph, we define a ``norm-like'' value for $g \in G$ as: \[
\| g\| := \inf \left\{ d_\Gamma(1,v)  : \exists v\in V\text{ and }g \in  vK\right\}.
\]
Note that $\|v\|=d_\Gamma(1,v)$ for $v\in V$.
We now define \[
\delta := \int_{g \in H} \exp(-2s_{0} \|g\|) \cdot \hat{\delta} \circ \rho(g)^{-1} \, d\mu.
\]

It suffices to check the four conditions in the proof of Theorem~\ref{thm:top}.
For the first item, we check\[\begin{aligned}
\int_{g \in H}  \exp(-2s_{0} \|g\|) \, d\mu &\le \int_{g \in G}  \sum_{v \in V} \chi_{g \in vK}\exp(-2s_{0} \|v\|) \, d\mu \\
&\le \sum_{v \in V} \int_{g \in G}\chi_{g \in vK}  \exp(-2s_{0} \|v\|) \, d\mu \\
&\le \sum_{v \in V}\exp(-2s_{0} \|v\|) < +\infty.
\end{aligned}
\]

For the third item, we use the inner regularity of $\mu$ to find a compact set $A \subseteq H$ such that \[
\int_{g \in H \setminus A} \exp(-2s_{0} \|g\|)\diam(X, \hat{\delta})\, d\mu < r/2. 
\]
We claim the existence of $r'>0$ such that for all $g \in A$ and for all $y \in \hat B(x, r')$ we have an ``equicontinuity at $x$'' in the sense that \[
\hat{\delta}(\rho(g^{-1}) (x), \rho(g^{-1})(y)) < \frac{r}{2 \mu(A)}.
\]
Indeed, if such an $r'$ does not exist, then there exist sequences $g_{n} \in A$ and $y_{n} \in X$ satisfying $\lim_{n} \hat{\delta}(x, y_{n})= 0$ and \[
\hat{\delta}(\rho(g_{n}^{-1}) (x), \rho(g_{n}^{-1})(y_n)) \ge \frac{r}{2\mu(A)}.
\]
We may assume that $g_n$ converges to some $g\in A$.
Let $N$ be a compact neighborhood of $x$ contained in $\rho(g)^{-1} (\hat B(x, r/2\mu(A)))$, the existence of which is given by the local compactness of $X$. Since $\rho$ is continuous, we have that $\rho(g_n)$ converges to $\rho(g)$ in $\Homeo(X)$.
It follows from the compact--open topology of $\Homeo(X)$ that $N \subseteq \rho(g_{n})^{-1} (\hat B(x, r/2\mu(A)))$ for sufficiently large $n$. Meanwhile, since $N$ is a compact neighborhood of $x$, the point $y_{n}$ eventually belongs to $N$ and hence to $\rho(g_{n})^{-1} (\hat B(x, r/2\mu(A)))$. This is a contradiction.

Given the previous claim,
let us assume $y\in \hat B(x,r')$ as above.
Then \[\begin{aligned}
{\delta}(x, y)& = \int_{g \in H\setminus A} \exp(-2s_{0}\|g\|) \hat{\delta}(\rho(g)^{-1} (x), \rho(g)^{-1} y)\, d\mu \\
&+ \int_{g \in  A} \exp(-2s_{0}\|g\|) \hat{\delta}(\rho(g)^{-1} (x), \rho(g)^{-1} y)\, d\mu \\
&\le \int_{g \in H \setminus A }\exp(-2s_{0}\|g\|)  \diam(X, \hat{\delta}) \, d\mu + \int_{g \in A} \hat{\delta}(\rho(g)^{-1} (x), \rho(g)^{-1} y)\, d\mu \\
&< \frac{r}{2} + \frac{r}{2\mu(A)} \cdot \mu(A) \le r
\end{aligned}
\]
as desired.

For the fourth item, we  claim that \[
\big|\|gh\| - \|g\| \big| \le 3\|h\|+1
\]
holds for all $g, h \in G$. To see this, let $\|g\| = m$ and $\|h\| = n$. Let $v_{1}, \ldots, v_{m} \in V$ be such that $v_{i-1}^{-1} v_{i} \in K^{3}$ for each $i$ and $v_{m}^{-1} g \in K$. Using a similar decomposition for $h$, one can take $u_{1}, \ldots, u_{3n+1} \in G$ such that $u_{i-1}^{-1}u_{i} \in K$ for each $i$ and $u_{3n+1 }= h$. Now for each $i=1, \ldots, 3n+1$, take $w_{i}\in V$ such that $w_{i}^{-1} \cdot (g u_{i}) \in K$. Then \[
(id, v_{1}, \ldots, v_{m}, w_{1}, \ldots, w_{3n+1})
\]
is a sequence in $V$, each step difference in $K^{3}$, such that $w_{3n+1}^{-1} gh \in K$. We conclude that $\|gh\| \le \|g\| + 3\|h\| +1$; the other inequality can be deduced similarly.

Given this estimate, we deduce that\[
\exp(-2s_0-6s_{0} \|g\|) \delta(x, y) \le (\rho(g) (x), \rho(g) (y)) \le \exp(2s_0+6s_{0} \|g\|) \delta(x, y)
\]
holds for any $x, y \in X$, completing the proof.
\end{proof}

\section*{Acknowledgements}
The authors are grateful to Ursula Hamenst{\"a}dt for the motivating idea and Nicolas Monod for guiding the authors to the reference~\cite{caprace2014on-embeddings}. The authors also thank the anonymous referee for helpful comments. The authors are supported by Mid-Career Researcher Program (RS-2023-00278510) through the National Research Foundation funded by the government of Korea, and by KIAS Individual Grants (Choi: SG091901 and Kim: MG073601) at Korea Institute for Advanced Study.
  
  \bibliographystyle{amsplain}
  \bibliography{ck2023}
  
\end{document}